\documentclass[12pt]{amsart}
\pdfoutput=1
\usepackage[english]{babel}
\usepackage[cp850]{inputenc}
\usepackage{amscd}
\usepackage{amssymb}
\usepackage{amsmath}
\usepackage{graphics}
\usepackage{graphicx}
\usepackage{hyperref}

\newtheorem{theo}{Theorem}

\newtheorem{lem}{Lemma}
\newtheorem{prop}{Proposition}

\newtheorem{rem}{Remark}

\newtheorem{con}{Conjecture}

\newcommand{\B}{\mathrm{bh}}

\author{}

\begin{document}

\begin{center}
{\bf\Large Algorithms for ball hulls and ball intersections in strictly convex  normed planes}\\[2ex]
by\\[2ex]
{\sc Pedro Mart\'{\i}n and Horst Martini }
\end{center}

\vspace*{3ex}

{\bf Abstract.}
{\small Extending results of Hershberger and Suri for the Euclidean plane, we show that
ball hulls and ball intersections of sets of $n$ points in strictly convex normed planes
can be constructed in $O(n \log n)$ time. In addition, we confirm that, like in the Euclidean subcase, the $2$-center problem
with constrained circles can be solved also for strictly convex normed planes in $O(n^2)$ time.
Some ideas for extending these results to more general types of normed planes are also presented.
}

\thanks{2000 {\it Mathematics Subject Classification}: 41A50, 46B20, 52A21, 52B55, 52C15, 68Q17.}

\thanks{{\it Key words and phrases}: ball hull, ball intersection, circumball, circumradius, Chebyshev set, Minkowski geometry, strictly convex normed plane, $2$-center problem}
\thanks{Research partially supported by MICINN (Spain) and FEDER (UE) grant MTM2008-05460,
and by Junta de Extremadura grant GR10060 (partially financed with FEDER).}

\title{}

\maketitle

\vspace{-5ex}

\section{Introduction}

The ball hull and the ball intersection of a given point set $K$ are common notions in
Banach-space theory; see, e.g., \cite{Bar-Pap}, \cite{Ma-Ma-Sp}, and \cite{T}. They denote intersections of congruent balls with
suitable radius which, in the first case, contain $K$ and, in the second one, have their centers in $K$.
Continuing algorithmical investigations of Hershberger and Suri (for the Euclidean subcase; see \cite{H-S}), we
present algorithmical approaches to the constructions of ball hulls and ball intersections of finite point sets $K$ in strictly
convex normed planes. Note that, although presenting only planar results, we stay with the notions of ball hull and ball intersection (instead of circular hull and circular
intersection) since they are common in this form. In other situations, we replace ``ball'' and ``sphere'' by \emph{disc} and \emph{circle}, respectively.
More precisely, we show that if $K$ consists of $n$ points, then the ball hull and the ball intersection of $K$ can be constructed in $O(n \log n)$ time.
For the case of ball hulls we additionally present  a second algorithm which is completely analogous to that from \cite{H-S}. We also discuss a further geometric question. The $2$-center problem asks for two closed discs  to cover $K$ (see \cite{Ag-Sh-SWe}, \cite{Ag-Pa-Av-Ri-Sh}, \cite{H-S-G}, \cite{H2}, and \cite{Sh}). Again generalizing results from \cite{H-S}, we show that the $2$-center problem with constrained center of suitably fixed radii can be solved in $O(n^2)$ time also if we extend it to strictly convex normed planes. In our final section, we present some results which can be taken as starting point for investigations of ball hulls in normed planes that no longer have to be strictly convex.

\medskip

Let $\mathbb{M}^d=(\mathbb{R}^d, \|\cdot\|)$ be a $d$-dimensional  \emph{normed} (or \emph{Minkowski}) \emph{space}. As well-known, the \emph{unit ball} $B$ of  $\mathbb{M}^d$ is   a compact, convex set with non-empty interior (i.e., a \emph{convex body}) centered at the \emph{origin} $o$. The \emph{boundary} of a closed set $A$ is denoted by $\partial A$, and  $\partial B$ is the \emph{unit sphere} of $\mathbb{M}^d$. Any homothetical copy $x+\lambda B$ of $B$  is called the  \emph{ball with center $x\in\mathbb{R}^d$ and radius} $\lambda > 0$ and  denoted by $B(x,\lambda)$; its boundary is the  \emph{sphere}  $S(x,\lambda).$  We use the usual abbreviation $\mathrm{conv}$ for \emph{convex hull}, and the \emph{line segment} connecting  the different points $p$ and $q$ is denoted by $\overline{pq}$, its affine hull is the  \emph{line}  $\langle p, q\rangle$. The vector $p-q$ is denoted by $\overrightarrow{qp}$.

Let $p$ and $q$  be two points of  the circle $S(x,\lambda)$ in $\mathbb{M}^2$. The \textit{minimal circular arc of $B(x,\lambda)$ meeting $p$ and $q$ } is the piece of $S(x,\lambda)$  with endpoints $p$ and $q$ which lies in the half-plane bounded by the line $\langle p, q\rangle$ and does not contain the center $x$. If $p$ and $q$ are opposite in $S(x,\lambda)$, then the two half-circles with endpoints $p$ and $q$ are minimal circular arcs of $S(x,\lambda)$ meeting $p$ and $q$. We denote a minimal circular arc meeting $p$ and $q$ by  $\widehat{pq}$.

Given a set $K$ of points in $\mathbb{M}^2$ and $\lambda>0$, the \emph{$\lambda$-ball hull} $\operatorname{bh}(K, \lambda)$ \emph{of} $K$ is defined as the intersection of all  balls of radius $\lambda$ that contain $K$:
$$
\operatorname{bh}(K, \lambda)=\bigcap_{K\subset B(x,\lambda)}B(x, \lambda).
$$

The \emph{$\lambda$-ball intersection} $\operatorname{bi}(K, \lambda)$ \emph{of} $K$ is the intersection of all  balls of radius $\lambda$ whose centers are from $K$:
$$
\operatorname{bi}(K, \lambda)=\bigcap_{x\in K}B(x,\lambda).
$$

 Of course,   these notions make only sense if $\operatorname{bi}(K, \lambda)\neq\emptyset$ and $\operatorname{bh}(K, \lambda)\neq\emptyset$. It is clear that $\operatorname{bh}(K, \lambda)\neq\emptyset$ if and only if $\lambda\geq\lambda_K$, where $\lambda_K$ is the smallest number such that $K$ is contained in a translate of $\lambda_K B$. Such a translate is called a \emph{minimal enclosing ball} (or \emph{circumball) of $K$}, and $\lambda_K$ is said to be the \emph{minimal enclosing radius } (or  \emph{circumradius } or  \emph{Chebyshev radius}) of $K$. Clearly, we have
 \begin{equation}K_1\subseteq K_2\;\Longrightarrow \lambda_{K_1}\leq \lambda_{K_2}.\label{11} \end{equation}
 In the Euclidean subcase the minimal enclosing ball of a bounded set is always unique, but this is no longer true for an arbitrary norm. It is  easy to check that
\begin{equation}\begin{split}\{x\in \mathbb{M}^d: x\;\text{is the center of a minimal enclosing disc of}\; K\}\\=\operatorname{bi}(K, \lambda_K),\end{split}\label{19}\end{equation} yielding that
  $\operatorname{bi}(K, \lambda)\neq\emptyset$ if and only if $\lambda\geq\lambda_K$. The set of centers of minimal enclosing balls of $K$ is called the \emph{Chebyshev set of $K$}.  Note that, in contrast to the Euclidean situation, in general normed spaces the Chebyshev set of a bounded set does not  necessarily belong  to the convex hull of this set (see \cite{Ma-Ma-Sp} for some examples).

For  a bounded compact set $K$ in $\mathbb{M}^d$  denote  by $\operatorname{diam} (K):=\max\{\|x-y\|: x, y\in K\}$  the \emph{diameter of } $K$.

 In what follows, when we speak about the $\lambda$-ball intersection or $\lambda$-ball hull of a set $K$, we always mean that $\lambda\geq\lambda_K$. It is easy to check that
\begin{equation}\label{12}\lambda_K\leq \mathrm{diam} (K)\leq 2 \lambda_K;
\end{equation}
see also \cite{Bar-Pap}.\\[0.2cm]

Note that for $d=2$ and $K$ a finite set, the boundary structure of $\mathrm{bi}(K,\lambda)$ consists of circular arcs of radius $\lambda$ with centers belonging to $K$. The following theorem  (see \cite{Ma-Ma-Sp}, \cite{Ma-Ma-Sp2}) describes the boundary structure of $\mathrm{bh}(K,\lambda)$.

\bigskip

\begin{theo}\label{theo1}
Let $K=\{p_1,p_2,\dots,p_n\}$ be a finite set in  a normed plane  $\mathbb{M}^2$,  and let $\lambda\geq \lambda_K$. We denote by  $\widehat{p_ip_j}$  a minimal circular  arc of radius $\lambda$ meeting $p_i$ and $p_j$. Let $\mathcal{H}$ be the set of all discs of radius $\lambda$ such that their boundary contains a circular  arc meeting points from $K$. If the plane $\mathbb{M}^2$ is strictly convex or $\lambda \geq \mathrm{diam}(K)$, then
$$
\operatorname{bh}(K,\lambda)=\bigcap_{K\subset B(x,\lambda)\in \mathcal{H}}B(x,\lambda)=\operatorname{conv}(\bigcup_{i,j=1}^n \widehat{p_ip_j}).
$$
\end{theo}

\bigskip

There exists a strong relationship between the ball hull and the ball intersection of a set $K$ as the following result shows (see \cite{Ma-Ma-Sp2}).
\begin{theo}\label{bi-bh}
 Let $K=\{p_1,p_2,\dots,p_n\}$ be a finite set in  a normed plane  $\mathbb{M}^2$ and $\lambda\geq \lambda_K$. If the plane $\mathbb{M}^2$ is strictly convex or $\lambda\geq \mathrm{diam}(K)$, then every arc of $\mathrm{bi}(K,\lambda)$ is generated by a vertex of $\mathrm{bh}(K,\lambda)$, and every vertex of $\mathrm{bi}(K,\lambda)$ is the center of an arc belonging to the boundary of $\mathrm{bh}(K,\lambda).$
\end{theo}

Both notions of ball hull and ball intersection are used for solving some versions of the 2-center problem in the Euclidean plane (see \cite{H-S}, \cite{H2})

Our paper is organized as follows:
\begin{itemize}
\item Section \ref{complexity ball} presents an algorithm for the ball intersection which takes $O(n \log n)$ time for strictly convex normed planes.
\item Section \ref{ball_intersection} shows an algorithm for the ball hull, taking $O(n \log n)$ time for strictly convex normed planes and based also on the results of Section \ref{complexity ball}.
\item Section \ref{2centerconstarinedcircles} contains  an algorithm for the 2-center problem with constrained circles, which takes $O(n^2)$ time for strictly convex normed  planes.
\item Section  \ref{algorithm ball hull idendical} yields  an algorithm for the ball hull, taking $O(n \log n)$ time for strictly convex normed planes and being identical to an algorithm of  Hershberger and Suri for the Euclidean subcase.
\item Section \ref{ball hull structure in a non strictly} contains  some results useful for studying the ball hull structure in a normed plane that is not necessarily  strictly convex.
\end{itemize}


\bigskip


\section{The complexity of an algorithm for $\mathrm{bi}(K,\lambda)$.}\label{complexity ball}
If $\lambda\geq \mathrm{diam}(K)$, then the centroid  $x=\frac{1}{n}\sum_{i=1}^n x_i$ belongs to $\mathrm{bi}(K,\lambda)$ and  is easy to locate. But if $\mathrm{diam}(K)>\lambda\geq \lambda_K$, it is not so obvious how to locate a point belonging to the ball intersection of $K$. For example, if $\lambda=\lambda_K$, then  the centroid is not necessarily a Chebyshev center (see \cite{Ma-Ma-Sp}).

We can  easily construct  the set $\mathrm{bi}(K,\lambda)$ ordering the points $\{p_1, p_2,\dots,p_n\}$ of $K$, starting with $\mathrm{bi}(\{p_1\},\lambda)$ and adding a point of $K$ in every step:

\begin{enumerate}
\item store $B(p_1,\lambda)$;
\item store $B(p_1,\lambda) \cap B(p_2,\lambda)$;
\item continue in the same way adding $p_3$, $p_4$, etc.
\end{enumerate}

Going this way, Hershberger and Suri  (see \cite{H-S}, Section 6.1, page 459) describe  an algorithm for computing $\mathrm{bi}(K,\lambda)$ in $O(n\log n)$ time for the Euclidean subcase.
They use this algorithm as a subroutine to solve the 2-center problem with centers at  points of $K$ (the 2-\emph{center problem}
\emph{with constrained circles}) in $O(n^2)$ time. In the present section, we rewrite the algorithm described in \cite{H-S} for a strictly convex normed plane, and  we use it in Section \ref{2centerconstarinedcircles} to solve the 2-center problem with constrained circles in  strictly convex normed  planes.

  Let us fix a Euclidean orthonormal system of reference   in the plane with basis $\{v_1,v_2\}$.  The points of a finite set   in this  plane can be ordered by their $x$-coordinates with respect to this basis, using the $y$-coordinate order for breaking the ties.

We consider the two lines parallel to the vector $v_2$ and supporting $\mathrm{bi}(K,\lambda)$,  and the  tangent points on $\partial \mathrm{bi}(K,\lambda)$ belonging to  them. The line meeting these two points separates  $\partial \mathrm{bi}(K,\lambda)$ in two components,  called \emph{upper chain} and \emph{lower chain} of $\partial \mathrm{bi}(K,\lambda)$.

We say that an arc $a_1$ is on the left with respect to the other arc $a_2$ if the leftmost point of $a_1$ has an $x$-coordinate smaller than the $x$-coordinate of the leftmost point of $a_2$, breaking the ties similarly as with the point order.

Every arc of $\mathrm{bi}(K,\lambda)$ has a center belonging to $K$. It is possible that some points of K are not centers of arcs of
$\mathrm{bi}(K,\lambda)$. The arcs of the upper (lower) chain can be ordered by this left-to-right order induced by their leftmost points. Hershberger and Suri state in \cite{H-S} that \emph{it is not difficult to see that the left-to-right order of the arcs along the upper (lower) chain of $\mathrm{bi}(K,\lambda)$  is just the reverse of the left-to-right order of the centers} (of these arcs), i.e., if $a_1,a_2,..,a_m$ is the ordered group
of arcs from left to right on the upper chain, and their centers are
$x_1,x_2,...,x_m$, respectively, then the centers $x_1,x_2,...x_m$ are ordered from right to
left. We prove in Lemma \ref{orderedarccenter} that this is also true for every strictly convex normed  plane.

Using a result of Gr\"{u}nbaum \cite{Grue1} and Banasiak \cite{Ban} (see Lemma \ref{twocircles} in Section \ref{ball hull structure in a non strictly}), in \cite{Ma-Ma-Sp2} the following lemma  is proved.

\begin{lem}\label{3.0}
Let $\mathbb{M}^2$ be a strictly convex normed plane  with unit disc $B$ and $p, q\in B$. Then the following statements hold true:
 \begin{enumerate}
 \item If  $p, q\in S(o,1)$ and there exists another circle $S(x,1)$ through $p$ and $q$, then $x=p+q$ and the origin $o$ and $x$ are in different half-planes with respect to the line $\langle p, q\rangle$.

 \item Any minimal circular arc of radius $1$ meeting $p$ and $q$ also belongs to $B$.

 \item  If a circular arc of radius 1 meeting $p$ and $q$ is  contained in $B$ such that it contains interior points of $B$, then this arc is a minimal circular arc.
 \end{enumerate}
\end{lem}

Discs of radius $1$ are considered only for simplicity; Lemma \ref{3.0}  is true for discs with arbitrary radius $\lambda$. This  lemma allows us to prove

\begin{lem}\label{orderedarccenter}
Let $\mathbb{M}^2$ be a strictly convex normed  plane. With the above conditions, if $K$ is a finite set in $\mathbb{M}^2$, then
the left-to-right order of the arcs along the upper (lower) chain of $\mathrm{bi}(K,\lambda)$  is just the reverse of the left-to-right order of the centers of these arcs.
\end{lem}
\begin{proof}
The upper and the lower chain cases are similar, and it is sufficient to prove the upper chain case. Let us fix a Euclidean orthonormal system of reference  with basis $\{v_1, v_2\}$. There exist two  lines parallel to $v_2$ supporting $\mathrm{bi}(K,\lambda)$, and every of them has a unique tangent point on $\partial \mathrm{bi}(K,\lambda)$. These two points determine the upper and the lower chain of $\mathrm{bi}(K,\lambda)$.

Let us consider a set $K$ of two points. Let $p^0$ be the leftmost point on $\mathrm{bi}(K,\lambda)$ with  respect to the system of reference. Namely, $p^0$ is the tangent point of a supporting line of $\partial \mathrm{bi}(K,\lambda)$ parallel to $v_2$ with the smallest first coordinate. The upper chain of $\partial \mathrm{bi}(K,\lambda)$ going clockwise is the part of $\partial \mathrm{bi}(K,\lambda)$ from $p^0$ to the other tangent point created by the parallel supporting line. The lower chain is the other part of $\partial \mathrm{bi}(K,\lambda)$.

\begin{figure}[ht]
\begin{center}
\includegraphics[width=10cm]{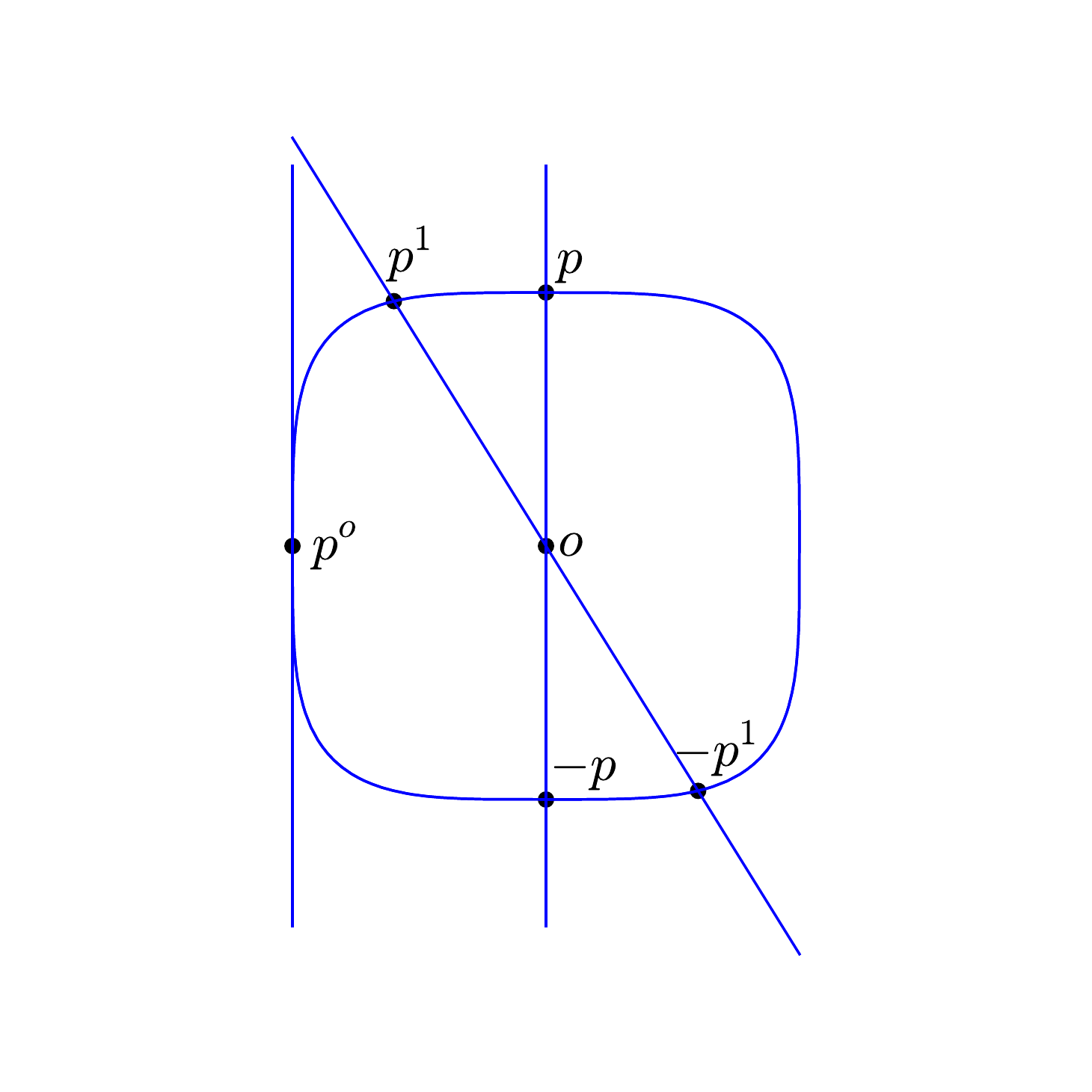}\\
\vspace*{-0.5cm}\caption{$p^1$ is between $p^o$ and $p$, clockwise.}\label{arcoscentros}
\end{center}
\end{figure}

Let $o$ and $x^1$ be the center of the first and the second arc, respectively, from left to right (in the arc order sense) over the upper chain of $\partial \mathrm{bi}(K,\lambda)$. Without loss  of generality, we can assume that the origin of the Euclidean  system of reference is $o$.

Moving from left to right along the upper chain of $\partial \mathrm{bi}(K,\lambda)$, let $p^1$ be the  vertex of $\partial\mathrm{bi}(K,\lambda)$ after $p^0$, namely $p^1\in S(o,\lambda)\cap S(x^1,\lambda).$ Let  $q^1$ be  the point such that $\{p^1,q^1\}=S(o,\lambda)\cap S(x^1,\lambda).$ The point $q^1$ is on $S(o,\lambda)$ between (clockwise) $p^1$ and $p^0$,  and by Lemma \ref{3.0}, $x=q^1+p^1$. 
 Let  $p$ be  the intersection \emph{upper} point between $S(o,\lambda)$ and the line parallel to $v_2$ containing $o$ (see Figure \ref{arcoscentros} and Figure \ref{arcoscentros2}). Since $\mathrm{diam}(\mathrm{bi}(K,\lambda))\leq 2\lambda$, the point $p^1$ is between $p^0$ and $-p^0$, clockwise. The following cases are possible.

Case 1: $p^1$ is on $S(o,\lambda)$ between $p^0$ and $p$, clockwise (as in  Figure \ref{arcoscentros}).

Subcase 1.1: $q^1$ is between $p^1$ and $-p^1$, clockwise. The arc $\widehat{p^1q^1}$, clockwise, is a minimal circular arc. Since  $x^1=q^1+p^1$, by the convexity of $\mathrm{bi}(K,\lambda)$ this situation is not possible.

Subcase 1.2: $q^1$ is between $-p^1$ and $p^0$, clockwise.  If $(p^1_1,p^1_2)$ and $(q^1_1,q^1_2)$ are the coordinates of $p^1$ and $q^1$, respectively, then $-p^1_1>q^1_1$. Therefore, $0>p^1_1+q^1_1$, and the point $x^1=p^1+q^1$ is on the left of $o$.


\begin{figure}[ht]
\begin{center}
\includegraphics[width=10cm]{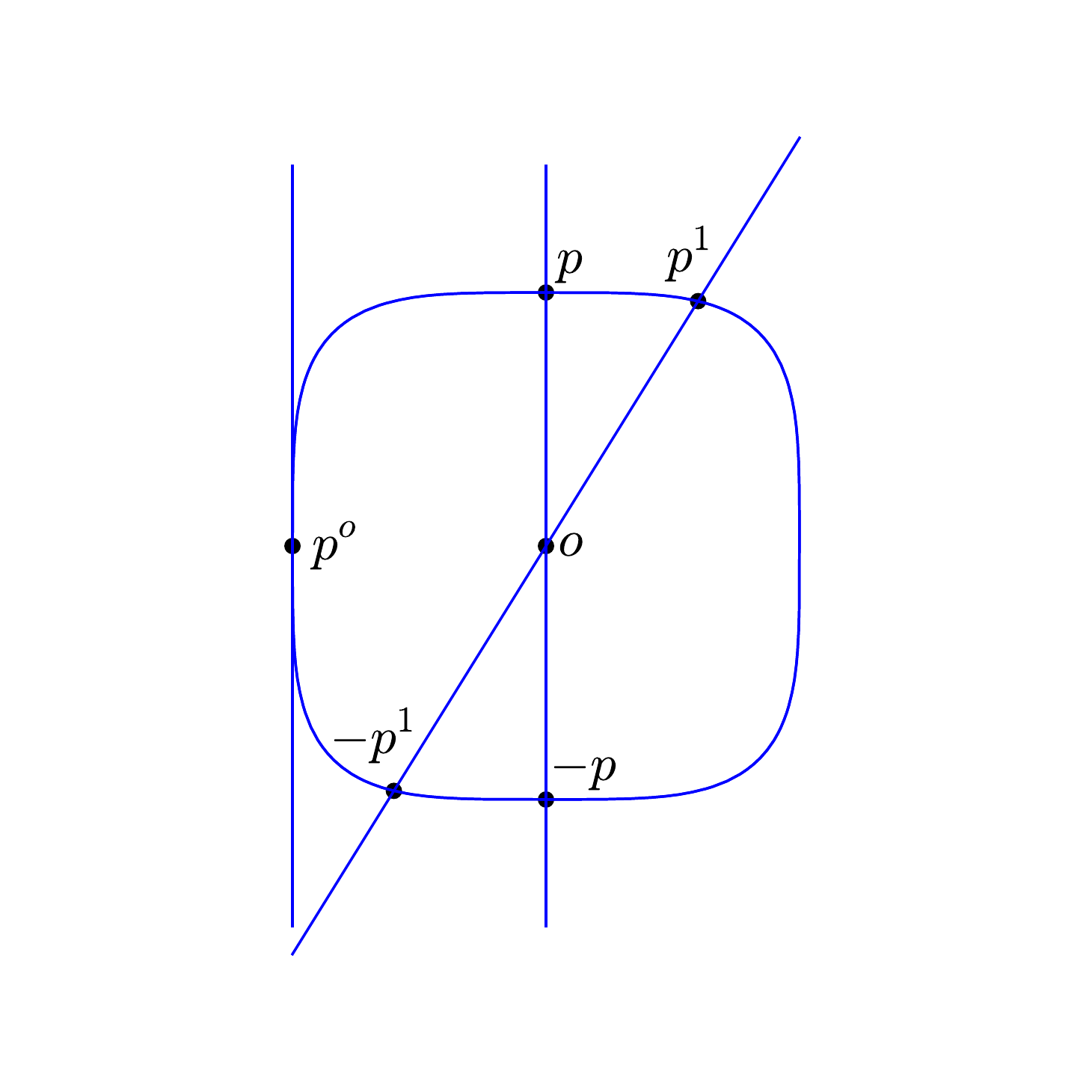}\\
\vspace*{-0.5cm}\caption{$p^1$ is between $p$ and $-p^0$, clockwise.}\label{arcoscentros2}
\end{center}
\end{figure}

Case 2: $p^1$ is on $S(o,\lambda)$ between $p$ and $-p^0$, clockwise (as in Figure \ref{arcoscentros2}).

Subcase 2.1: $q^1$ is between $p^1$ and $-p^1$, clockwise. The arc $\widehat{p^1q^1}$, clockwise, is a minimal circular arc. Since  $x^1=q^1+p^1$, by the convexity of $\mathrm{bi}(K,\lambda)$ this situation is not possible.

Subcase 2.2: $q^1$ is between $-p^1$ and $p^0$, clockwise. If $(p^1_1,p^1_2)$ and $(q^1_1,q^1_2)$ are the coordinates of $p^1$ and $q^1$, respectively, then $-p^1_1>q^1_1$. Therefore, $0>p^1_1+q^1_1$, and the point $x^1=p^1+q^1$ is on the left of $o$.

If $K$ is a set of $n$ points,  again $p^0$ is the leftmost point on $\mathrm{bi}(K,\lambda)$ with respect to the system of reference;  $p^1,p^2,\dots p^{m-1}$ are  the following vertices on the upper chain of $\partial \mathrm{bi}(K,\lambda)$, clockwise; $p^m$ is the rightmost point on $\mathrm{bi}(K,\lambda)$; $o$ is  the center of the arc $\widehat{p^0p^1}$; and $x^1,x^2,\dots x^m$ are the centers of the left-to-right ordered arcs $\widehat{p^1p^2},\widehat{p^2p^3},\dots,\widehat{p^{m-1}p^{m}}$, respectively.

We have proved the statement for a set $K$ containing two points. But if the set $K$ has $n$ points and a piece of $S(x^i,\lambda)\cap S(x^{i+1},\lambda)$ belongs to the upper chain of $\partial \mathrm{bi}(K,\lambda)$, then this piece also belongs to the upper chain of  $\mathrm{bi}(\{x^i,x^{i+1}\},\lambda)$, and their common arcs are located in the same arc order. Therefore, we can  repeatedly apply  the  statement  proved for two points to the pairs $(x^i,x^{i+1})$ and justify that the centers $x^1,x^2,\dots x^m$ are ordered conversely to the  sequence of  the arcs $\widehat{p^1p^2},\widehat{p^2p^3},\dots,\widehat{p^{m-1}p^{m}}$.
\end{proof}

After sorting the points of $K$ by the $x$-coordinate, it is easier and cheaper to build $\mathrm{bi}(K,\lambda)$, because starting with the leftmost arc and its center, one  only has to consider the centers  at the left side to find the following arc at the right one. Therefore, the upper (lower) chain of $\mathrm{bi}(K,\lambda)$ can be constructed in $O(n)$ time, as Hershberger and Suri describe: \emph{If a new circle contributes to the chain at all, its arc appears  at the left end of the chain, possibly removing  some previously added arcs. Computing the new boundary takes constant time, plus time proportional to the number of arcs deleted. Hence the overall bound} (of building $\mathrm{bi}(K,\lambda)$) \emph{is $O(n)$.}

\begin{theo}\label{algorithmballintersection}
Let $\mathbb{M}^2$ be a strictly convex normed plane. If $K$ is a set of $n$ points and $\lambda\geq\lambda_K$, then the set $\mathrm{bi}(K,\lambda)$ can be constructed  via an algorithm taking $O(n\log n)$ time.

\end{theo}
\begin{proof}
Sorting the points of $K$ from left to right takes $O(n\log n)$ time. After the points are ordered, constructing  $\mathrm{bi}(K,\lambda)$ takes $O(n)$ time. Therefore, the total cost is $O(n\log n)$ time.
\end{proof}

\section{A ball hull algorithm based in the ball intersection algorithm}\label{ball_intersection}
 In the proof of Proposition 5.5 in \cite{Ma-Ma-Sp}  an algorithm for building $\mathrm{bh}(K,\lambda)$ in any normed plane valid for the case $\lambda\geq \mathrm{diam}(K)$ is implicitly described. It starts  with a point $x$ such that $K\subset B(x,\lambda)$. For example, the centroid of $K$ is a useful starting point.  This Proposition 5.5 in \cite{Ma-Ma-Sp} is extended in \cite{Ma-Ma-Sp2} for $\lambda \geq \lambda_K$ when the plane is strictly convex.
 But, unfortunately and as we noted above, it is not easy to locate the starting point for the algorithm when $\lambda< \mathrm{diam}(K)$.

 Nevertheless, we can first construct  the set $\mathrm{bi}(K,\lambda)$ as in Section \ref{complexity ball} in $O(n\log n)$ time, and after that we  carry on the steps described in \cite{Ma-Ma-Sp2} for building $\mathrm{bh}(K,\lambda)$. We develop this idea in the present section.

\begin{theo}\label{bi-bh2}
Let $K=\{p_1,p_2,\dots,p_n\}$ be a finite set in  a normed plane  $\mathbb{M}^2$ and $\lambda\geq \lambda_K$. Let $K'$ and $K''$ denote the set of vertices of $\mathrm{bi}(K,\lambda)$ and the set of vertices of $\mathrm{bh}(K,\lambda)$, respectively. If either $\lambda\geq \mathrm{diam}(K)$ or $\mathbb{M}^2$ is strictly convex, then
\begin{align}
\mathrm{bh}(K,\lambda)&= \mathrm{bi}(K',\lambda),\label{1}\\
\mathrm{bi}(K,\lambda)&= \mathrm{bi}(K'',\lambda).\label{2}
\end{align}
Furthermore, if $\mathbb{M}^2$ is strictly convex, the left-to-right order of the arcs along the upper (lower) chain of $\mathrm{bh}(K,\lambda)$  is just the reverse of the left-to-right order of the centers of these arcs (which belong to $K'$).
\end{theo}

 \begin{proof}
From Theorem \ref{bi-bh} (and from its proof presented in \cite{Ma-Ma-Sp2}) one  can  deduce  (\ref{1}) and (\ref{2}). Using Lemma \ref{orderedarccenter}, we have the last  statement.

\end{proof}




Therefore, having obtained $\mathrm{bi}(K,\lambda)$, one can construct $\mathrm{bh}(K,\lambda)$ in a strictly convex normed plane by plotting the arcs with centers in the vertices of $\mathrm{bi}(K,\lambda)$, describing finally the following algorithm:

\begin{enumerate}
\item Sorting the points of $K$ in $O(n\log n)$ time.
\item Building $\mathrm{bi}(K,\lambda)$ in $O(n)$ time (Theorem \ref{algorithmballintersection}).
\item Considering  the set $K'$ of sorted vertices $\{x_1,...,x_k\}$ of $\mathrm{bi}(K,\lambda)$ obtained in (2).
\item Building $\mathrm{bh}(K,\lambda)=\mathrm{bi}(K',\lambda)$ (Theorem \ref{bi-bh2}) in $O(n)$ time (Theorem \ref{algorithmballintersection}).
\end{enumerate}

\begin{theo}\label{algorithmballhull}
Let $\mathbb{M}^2$ be a strictly convex normed plane. If $K$ is a set of $n$ points and $\lambda\geq\lambda_K$, then the set $\mathrm{bh}(K,\lambda)$ can be constructed via an algorithm taking $O(n\log n)$ time.
\end{theo}


\section{The 2-center problem with constrained circles}\label{2centerconstarinedcircles}

 The 2-center problem with constrained circles asks the following: given a (finite) set $K$ of points in the plane, one has to find two closed discs of suitably fixed radii whose centers belong to $K$ and whose union covers $K$. Let us assume that the fixed radii are $r$ and $1$ ($r\geq 1$). An algorithm by Hershberger and Suri (\cite{H-S}, page 459) solves this problem in the Euclidean plane taking $O(n^2)$ time in the following way:
\begin{enumerate}
\item Sorting the points of $K$ from left to right in $O(n \log n)$ time, according to  the $x$-coordinate.

\item For each point $p\in K$:
    \begin{enumerate}
    \item Determining the set $U$ of ordered points whose distance from $p$ is greater than the radius $r$.
    \item Afterwards, obtaining $\mathrm{bi}(U,1)$, which takes $O(n)$ time (see Section 6.1 in \cite{H-S}
    or Section \ref{complexity ball} in the present paper).
    \item Testing if $\mathrm{bi}(U,1)$ contains some point of $K$. This  can also be done in $O(n)$ time, "\emph{marching through $K$ from left to right, maintaining the two arcs of $\partial \mathrm{bi}(U,1)$ that overlap the $x$-coordinate of the current point. We sweep over each arc once, so the total cost is linear}".
        \end{enumerate}

\end{enumerate}
Therefore, the total time for solving this 2-center problem with constrained circles is $O(n \log n)+ n \cdot O(n)=O(n^2)$.

\begin{theo}
Let $\mathbb{M}^2$ be a strictly convex normed plane. If $K$ is a set of $n$ points and $r\geq 1$, the 2-center problem with constrained circles centered in the points of $K$ can be solved in $O(n^2)$ time.
\end{theo}

\begin{proof}
The 2-center problem with constrained circles for strictly convex normed planes can be solved with the same algorithm as presented in \cite{H-S} using Theorem \ref{algorithmballintersection} in Section \ref{complexity ball} for the ball intersection.
\end{proof}

\section{A ball hull algorithm identical to Hershberger-Suri's one}\label{algorithm ball hull idendical}
 Hershberger and Suri (\cite{H-S}, p. 443) presented an algorithm constructing the ball hull of $n$ points in $O(n \log n)$ time in the  Euclidean plane. In this section, we adapt this algorithm for  strictly convex normed  planes.

The following result is proved in \cite{Ma-Ma-Sp2}.
\begin{lem}\label{4.0}
Let $\mathbb{M}^2$ a normed plane. Let $p$ and $q$ be two points belonging to a disc of radius $\lambda$. If $\mathbb{M}^2$ is strictly convex or $\|p-q\|\leq \lambda$, then
\begin{enumerate}
\item there exist only two minimal arcs of radius $\lambda$ meeting $p$ and $q$ (which may degenerate to the segment $\bar{pq}$),
\item every disc of radius $\lambda$ containing $p$ and $q$ also contains the minimal circular arcs of radius $\lambda$ meeting $p$ and $q$,
\item for every $\alpha\geq \lambda$, each disc of radius $\lambda$ containing $p$ and $q$ also contains the minimal circular arcs of radius $\alpha$ meeting $p$ and $q$,
\item if a circular arc of radius $\lambda$ meeting $p$ and $q$ is contained in a disc of radius $\lambda$ such that it contains interior points of the disc, then this arc is a minimal circular arc.
\end{enumerate}

\end{lem}

Only for simplicity, we often use in this section discs of radius $1$ and denote by $\B(K)$ to the  ball hull  with  radius 1 of the set $K$.

\begin{lem} [Lemma 4.9 in \cite{H-S}]\label{lem4.9}
Let $A$ and $B$ be disjoint sets of points. If $v\in A$ is from $\partial \mathrm{bh}(A\cup B)$, then $v$ is also from $\partial \mathrm{bh}(A).$ If $v_1,v_2\in A$ are consecutive vertices of $\partial \mathrm{bh}(A\cup B)$, then there exists a minimal arc\footnote{In \cite{H-S}, the statement is formulated \emph{"..., then the arc between them appearing..."}, but there exist two minimal arcs of radius $\lambda$ meeting every pair of points.} of radius one meeting $v_1$ and $v_2$ appearing on both $\partial \mathrm{bh}(A\cup B)$ and $\partial \mathrm{bh}(A).$
\end{lem}
\begin{proof}
Let $v\in A\cap \partial \mathrm{bh}(A\cup B)$. Since $A \subseteq \mathrm{bh}(A)\subseteq \mathrm{bh}(A\cup B)$, every point belonging to the interior of $\mathrm{bh}(A)$ is a point of the interior of $\mathrm{bh}(A\cup B)$. Therefore, $v$ is not an interior point of $\mathrm{bh}(A)$.

The second part of the statement holds because of the boundary structure of the ball hull (Theorem \ref{theo1}) and $A \subseteq \mathrm{bh}(A)\subseteq \mathrm{bh}(A\cup B).$
\end{proof}

The data structure (for storing the information about points, vertices of the ball hull and so on) in \cite{H-S} is introduced in p. 444 and extended in p. 446. The same structure is also valid in a strictly convex normed plane. The points of an input set $K$ are ordered by their $x$-coordinates, and  a complete binary tree 
 $T(K)$ is used to organize them. The leaves of $T(K)$ are the ordered points of $K$. Every node of $T(K)$ represents the ball hull of the points in the leaves of its subtree. Therefore, the root of $T(K)$ represents the ball hull of the points $K$. In fact, every node represents a ball hull, and this information is stored like a doubly linked list of its vertices such that for every vertex the predecessor and the successor is known. Since a point
  can be the vertex of more than one ball hull, for economizing space every point is only stored as vertex at the highest level in the tree at which it appears on a ball hull.

Let $F$ be the set of points represented by a node in the tree. We denote by $L$ and $R$ the sets of points associated to the left and right children nodes, respectively, of the $F$-node.

\begin{lem} [Lemma 4.10 in \cite{H-S}] \label{lem4.10}
Let $L$ and $R$ be two finite sets of  points separable by a line $x=x_{split}.$ At most one arc of $\B (L)$ and one arc of $\B (R)$ cross the vertical line $x=x_{split}$. $\B (L)$ contains $R$ if and only $\B (L)$ contains $\B (R)$ (and vice versa in the symmetric sense).
\end{lem}

\begin{proof}
For strictly convex normed planes, the proof from \cite{H-S} is sufficient by using  Lemma \ref{4.0}.
\end{proof}

\begin{lem} [Lemma 4.11 in \cite{H-S}] \label{lem4.11}
Let $L$ and $R$ be two finite sets of  points separable by a line $x=x_{split}.$ If neither $\B (L)$ nor $\B (R)$ contains the other, then $\partial \B (L\cup R)$ contain two arcs not in $\partial \B (L)$ or $\partial \B (R)$, both crossing $x=x_{split}$. The other arcs in $\B (L\cup R)$ form two continuous chains, one from $\partial \B (L)$ and one from $\partial \B (R)$.
\end{lem}

\begin{proof}
The proof in \cite{H-S} is also valid for strictly convex normed planes.
\end{proof}
The two arcs described in the previous lemma are named the \emph{outer common tangents}, in the circular sense, of $L$ and $R$.

Given a set of points $F$ represented by a node in the tree, it is possible to build $\mathrm{bh}(F)$ from the points $L$ and $R$ associated to the left and right children nodes, respectively, of the $F$-node following this way:
\begin{enumerate}
\item \emph{Determining whether either $\mathrm{bh}(L)$ contains $\mathrm{bh}(R)$ or vice versa, and finding their common tangents if neither does}. In Lemma 4.12 in \cite{H-S}  a technique  is presented  doing this in linear time in the sizes of the lists of vertices involved. We rewrite this lemma for every strictly convex normed  plane in the form of Lemma \ref{lem4.12} below.
\item \emph{Updating the list of vertices of $\mathrm{bh}(F)$ once the tangents are known}. This takes constant time.
\end{enumerate}

The construction of $\mathrm{bh}(K)$ uses a divide-and-conquer algorithm. Every node is obtained from the combination of the two children, starting from the leaves of the tree.

\begin{lem} [Lemma 4.12 in \cite{H-S}]\label{lem4.12}
The data structure $T(K)$ can be built in $O(n \log n)$ time.
\end{lem}
\begin{proof}
The proof of this lemma in \cite{H-S} is valid for strictly convex planes. It is based on the following:

 1) \emph{Determining whether either $\mathrm{bh}(L)$ contains $\mathrm{bh}(R)$ or vice versa, takes linear time.}

 $\mathrm{bh}(L)$ contains $\mathrm{bh}(R)$ if and only if the vertices of $R$ are contained in $\mathrm{bh}(L)$, and if and only if these vertices are contained in the rightmost arc of $\mathrm{bh}(L)$. To check this will take $O(|R|)$ time. Similarly for the reciprocal. If neither contains the other, then, by Lemma \ref{lem4.11}, common tangents exist.

 2) \emph{Finding the common tangents of  $\mathrm{bh}(L)$ and $\mathrm{bh}(R)$ takes,  if they exist, $O(|L|+|R|)$ time.}

 Let $l_0$, $l_1$, $l_2$ be three consecutive (counterclockwise) vertices of $\mathrm{bh}(L)$. Let us consider the minimal arcs $\widehat{l_0l_1}$ and $\widehat{l_1l_2}$ on $\partial\mathrm{bh}(L)$ and the unit discs $B_{0,1}$ and $B_{1,2}$ whose circles contain these arcs. If we move (counterclockwise) the center of the disc $B_{0,1}$  along $S(l_1,1)$ to get the center of the disc $B_{1,2}$, then every unit circle "between" them throughout this movement contains $\mathrm{bh}(L)$ (see the constructive proof of Theorem \ref{theo1} in \cite{Ma-Ma-Sp2}). Furthermore, as it is explained in the proof of Lemma 9 in \cite{Ma-Ma-Sp2}, the circle  "between" (counterclockwise) $B_{0,1}$ and $B_{1,2}$ cannot simultaneously contain points of both the arcs which form the boundary of $B_{0,1}\cap B_{1,2}$, apart from $l_1$. In the proximity of $l_1$, every circle of these discs  has a branch belonging to $B_{0,1}\setminus B_{1,2}$,  and a  branch belonging to $B_{1,2}\setminus B_{0,1}$

Let $v_1$ be the rightmost point of $R$. To find the common tangents, we pick a vertex $l_1$ of $\mathrm{bh}(L)$ and find the unit circle $B_{0,1}$, which is tangent to $\mathrm{bh}(L)$ at $l_1$. Then we roll $B_{0,1}$ around the boundary of $\mathrm{bh}(L)$ until 1) it passes through $v_1$, and 2) the upper arc from $v_1$ to $\mathrm{bh}(L)$ (counterclockwise) is a minimal arc. For this purpose, it is sufficient to check the
conditions 1) and 2) for the arcs $\widehat{v_1l_i}$, with $l_i\in l$. It takes $O(|L|)$ time.

For simplicity, we assume that the arc $\widehat{v_1l_1}$ satisfies conditions 1) and 2). Starting with the circle which contains $\widehat{v_1l_1}$, we roll it (similarly as above)  around $\mathrm{bh}(L)$ counterclockwise, preserving fixed its $\mathrm{bh}(L)$-vertex tangent point of the circle while it is possible, until it becomes tangent to $\mathrm{bh}(R)$. Let us denote by  ${v_2,...,v_r}$ the (counterclockwise) consecutive vertices of $\mathrm{bh}(R)$ following after $v_1$ (similar notation for $l_2,...,l_s$). The upper common tangent is one of the arcs meeting a pair of points $(v_j,l_i)$, and we can find it starting with $\widehat{v_1,l_1}$ and moving the set of vertices $v_j$ counterclockwise until an arc $\widehat{v_jl_1}$ is not tangent to $\mathrm{bh}(L)$ for some $j$. If this happens, we continue with the arcs $\widehat{v_jl_2}$ increasing the index of the vertices $v_j$ but starting from the last index $j$ reached. Continuing with this process, a point $l_i$ is discarded in a step when an arc $\widehat{v_jl_i}$ is not tangent to $\mathrm{bh}(L)$ for some $j$. If $l_i$ is discarded in a step $\widehat{v_k,l_i}$, then $l_i$ is discarded definitely, and we continue the search with the arc $\widehat{v_{k},l_{i+1}}$. Therefore, we need at most $|L|+|R|$ steps, and so it takes $O(|L|+|R|)$ time to
 locate the upper common tangent.

 3) \emph{Updating the list of vertices of $\mathrm{bh}(F)$ when  the tangents are known takes constant time}.

It is possible to build the list for the root of $T(K)$ building recursively the lists for the children and combining them later. The running time of the algorithm is given by the recurrence $f(n)=2f(n/2)+O(n)$, which has the solution $f(n)=O(n \log n).$
\end{proof}
As a consequence of Lemma \ref{lem4.12}, we obtain the following Theorem.
\begin{theo}
Let $\mathbb{M}^2$ be a strictly convex normed plane. If $K$ is a set of $n$ points, then the set $\mathrm{bh}(K,\lambda)$ can be built with a divide-and-conquer algorithm taking $O(n\log n)$ time.
\end{theo}

\begin{rem}
\cite{Ma-Ma-Sp2} and Lemma \ref{4.0} are also useful in order to prove the rest of Hershberger's and Suri's results in section 4.3 of the same paper \cite{H-S}, which allow them to solve  the  2-center problem in the Euclidean plane (for the constrained subcase in $O(n^2)$ time, and for the unconstrained subcase in $O(n^2 \log n)$ time). This 2-center problem for strictly convex normed planes is studied in section \ref{2centerconstarinedcircles} (constrained subcase) in the present paper and  in \cite{Ma-Ma-Sp2} (unconstrained subcase).

\end{rem}
\section{The ball hull structure in a  more general setting}\label{ball hull structure in a non strictly}

If either $\lambda\geq \mathrm{diam}(K)$ or $\mathbb{M}^2$ is strictly convex, Theorem \ref{theo1} 
 describes the boundary structure of $\mathrm{bh}(K,\lambda)$. For its proof (see \cite{Ma-Ma-Sp} and \cite{Ma-Ma-Sp2}) it is necessary to use that every ball of radius $\lambda$ containing a pair of points of $K$ always contains all minimal arcs of radius larger than or equal to $\lambda$ meeting this pair of points (Lemma 4 in \cite{Ma-Ma-Sp2}). But this is not true in a normed plane which is not  strictly convex whether $\mathrm{diam}(K)>\lambda\geq \lambda_K$ (for instance, when we consider the maximum  norm). In order to describe the ball hull structure in a general normed plane, it is interesting to clarify what happens with these $\lambda$-minimal arcs meeting two points contained in a ball of the same radius.

Gr\"{u}nbaum \cite{Grue1} and Banasiak \cite{Ban} proved the following lemma (see also \cite[$\S$ 3.3]{MSW}).

\begin{lem}\label{twocircles} Let $\mathbb{M}^2$ be a normed plane. Let $C\subset \mathbb{M}^2$ be a compact, convex disc whose  boundary is the closed curve $\gamma$; $v$ be a vector in $ \mathbb{M}^2$; $C'=C+v$ be a translate of $C$ with boundary $\gamma'$. Then $\gamma\cap \gamma'$ is the  union of two segments, each of which may degenerate to a point or to the empty set.

Suppose that this intersection consists of two connected non-empty components $A_1$, $A_2$. Then the two lines parallel to the line
of translation and supporting $C\cap C'$ intersect $C\cap C'$  exactly in $A_1$ and $A_2.$

Choose a point $p_i$ from each component $A_i$ and let $c_i=p_i-v$ and $c_i'=p_i+v$ for $i=1,2.$ Let $\gamma_1$ be the part of $\gamma$ on the same side of the line $\langle p_1, p_2\rangle$ as $c_1$ and $c_2$; let $\gamma_2$ be the part of $\gamma$ on the  side of  $\langle p_1, p_2\rangle$ opposite to  $c_1$ and $c_2$; and similarly for $\gamma'$, $\gamma_1'$, and $\gamma_2'$.

Then $\gamma_2\subseteq \operatorname{conv}(\gamma_1')$ and $\gamma_2'\subseteq \operatorname{conv}(\gamma_1).$
\end{lem}

Let $\mathbb{M}^2$ be a normed plane which is not strictly convex. Let $o$ be the origin of the plane, and $x$, $p$ and $q$ be points  such that $p,q\in S(o,\lambda)\cap S(x,\lambda)$. Let us consider both minimal arcs of $S(o,\lambda)$ and $S(x,\lambda)$ meeting $p$ and $q$. There are different situations:

Case 1. \textsl{The centers $o$ and $x$ belong to different half-planes bounded by the line $\langle p, q \rangle$.} By Lemma \ref{twocircles}, the minimal arc meeting $p$ and $q$ bounded by $S(x,\lambda)$ is contained in $B(o,\lambda)$.

Case 2. \textsl{The centers $o$ and $x$ belong to the same half-plane bounded by the line $\langle p, q \rangle$.}
The points $x$, $o$ and $p+q$ belong to $S(p,\lambda)\cap S(q,\lambda)$. By Lemma \ref{twocircles}, $S(p,\lambda)\cap S(q,\lambda)$ is the  union of two segments, each of which may degenerate to a point or to the empty set.

Subcase 2.1. \textsl{$S(p,\lambda)\cap S(q,\lambda)$ are two different non-empty, connected components.} The two lines  parallel to $\vec{pq}$ and supporting $S(p,\lambda)\cap S(q,\lambda)$ intersect $S(p,\lambda)\cap S(q,\lambda)$  exactly in both connected components. Due to the fact that $x$ and $o$ are in the same half-plane defined by $\langle p, q \rangle$, they belong to the same connected component of the intersection, and the vector $\vec{ox}$ is parallel to $\vec{pq}.$ Without loss of generality let us assume that $\vec{ox}=\alpha \vec{qp}$ with $\alpha>0.$

The point $x+\vec{op}$ belongs to the line $\langle p,q\rangle$, and it is at distance $\lambda$ from $x$. Therefore, the segment meeting $x+\vec{op}$, $p$ and $q$ belongs to $S(x,\lambda)$, and the segment $\overline{pq}$ is the minimal arc meeting $p$ and $q$ from  the circle $S(x,\lambda).$ By convexity, this segment belongs to the disc $B(o,\lambda).$

Subcase 2.2. \textsl{$S(p,\lambda)\cap S(q,\lambda)$ is only one connected  component.} The points $x$, $o$ and $p+q$ belong to this connected component, and they are aligned. Let us assume that $o$ is situated in this segment-component between $x$ and $p+q$. The line $\langle p, q \rangle$ separates the points $x$ and $o$ from $p+q$. By Lemma \ref{twocircles}, the minimal arc meeting $p$ and $q$ defined by $S(x,\lambda)$ is contained in the disc $B(o,\lambda)$. Let $a$ be the extreme point of the segment $S(p,\lambda)\cap S(q,\lambda)$ such that $x$ is situated between $a$ and $o$. By the same previous argument, the minimal arc meeting $p$ and $q$ from $S(a,\lambda)$ is contained in  every disc with center belonging to the part
of $S(p, \lambda)\cap S(q, \lambda)$ which is situated in the half-plane defined by the
line $\langle p,q\rangle$ and containing $o$.

With the above situation and the results included in \cite{Ma-Ma-Sp2} (see Lemma 3 and Lemma 4 there), we can prove the following.

\begin{prop}\label{3.0extended}
Let $\mathbb{M}^2$ be a normed plane. For every pair of points $p$ and $q$ whose distance is less or equal to $2\lambda$, there exist two minimal circular arcs meeting them (eventually only one, if they degenerate to the same segment) which belong to every disc of radius $\lambda$ containing $p$ and $q$. These two arcs (if they are really two) are situated in different half planes  bounded by the line $\overline{pq}.$ The centers of the discs defining these two minimal arcs are situated in the extreme points of the connected components $S(p,\lambda)\cap S(q,\lambda)$.
\end{prop}

Equipped with Proposition \ref{3.0extended}, it seems possible to obtain the following results (proved in \cite{Ma-Ma-Sp} and  \cite{Ma-Ma-Sp2} whether either $\lambda\geq \mathrm{diam}(K)$ or $\mathbb{M}^2$ is strictly convex), as well as extend some of the algorithms presented in this paper  to a general normed plane.

\begin{con}\label{con1}
Let $K=\{p_1,p_2,\dots,p_n\}$ be a finite set in a  normed plane $\mathbb{M}^2$, and $\lambda\geq \lambda_K$. Then
$$\operatorname{bh}(K)=\bigcap_{i=1}^k B(x_i,\lambda),$$
where $B(x_i,\lambda)$, $i=1,2,...,k$, are balls which contain $K$, and their spheres contain some minimal arcs meeting points of $K$.
\end{con}

\begin{con}\label{con2}
 Let $K=\{p_1,p_2,\dots,p_n\}$ be a finite set in  a normed plane  $\mathbb{M}^2$ and $\lambda\geq \lambda_K$. Then every arc of $\mathrm{bi}(K,\lambda)$ is generated by a vertex of $\mathrm{bh}(K,\lambda)$, and every vertex of $\mathrm{bi}(K,\lambda)$ is the center of an arc belonging to the boundary of $\mathrm{bh}(K,\lambda).$
\end{con}

\vspace{1cm}
\begin{tabular}{l}
Pedro Mart\'{\i}n\\
Departamento de Matem\'{a}ticas,\\
Universidad de Extremadura,\\
06006 Badajoz, Spain\vspace{0.1cm}\\
E-mail: pjimenez@unex.es
\end{tabular}\vspace{0.3cm}

\begin{tabular}{l}
Horst Martini\\
Fakult\"at f\"ur Mathematik, TU Chemnitz\\
D-09107 Chemnitz, Germany\vspace{0.1cm}\\
E-mail: horst.martini@mathematik.tu-chemnitz.de
\end{tabular}\vspace{0.3cm}

\end{document}